\documentclass[11pt, draft]{amsart}
\usepackage{amssymb, amstext, amscd, amsmath, amssymb}
\usepackage{mathtools, color, paralist, dsfont, rotating}
\usepackage{verbatim}
\usepackage{enumerate}
\usepackage{tikz-cd}
\usepackage{tikz}
\usepackage[all]{xy}
\usepackage{float}
\numberwithin{equation}{section}

\makeatletter
\def\@cite#1#2{{\m@th\upshape\bfseries%
[{#1\if@tempswa{\m@th\upshape\mdseries, #2}\fi}]}}
\makeatother
\theoremstyle{plain}
\newtheorem{theorem}{Theorem}[section]
\newtheorem{corollary}[theorem]{Corollary}
\newtheorem{proposition}[theorem]{Proposition}
\newtheorem{lemma}[theorem]{Lemma}
\theoremstyle{definition}
\newtheorem{definition}[theorem]{Definition}
\newtheorem{example}[theorem]{Example}

\newtheorem{remark}[theorem]{Remark}

\newtheorem{hypothesis}{Hypothesis}

\theoremstyle{remark}

\mathtoolsset{centercolon}
\renewcommand{\H}{{\mathcal{H}}}

\newcommand{\K}{{\mathcal{K}}}

\renewcommand{\O}{{\mathcal{O}}}

\newcommand{\T}{{\mathcal{T}}}

\newcommand{\rS}{\mathrm{S}}

\newcommand{\bF}{\mathbb{F}}
\newcommand{\bN}{\mathbb{N}}
\newcommand{\bT}{\mathbb{T}}

\newcommand{\fS}{{\mathfrak{S}}}

\newcommand{\fZ}{{\mathfrak{Z}}}

\newcommand{\wot}{\textsc{wot}}

\DeclareMathOperator*{\sotsum}{\textsc{sot--}\!\!\sum}

\newcommand{\Alg}{\operatorname{Alg}}

\newcommand{\Span}{\operatorname{span}}

\begin{document}

\title[Self-adjoint free semigroupoid algebras]{All finite transitive graphs admit a self-adjoint free semigroupoid algebra}

\author[A. Dor-On]{Adam Dor-On}
\address{Department of Mathematical Sciences \\ University of Copenhagen \\ Copenhagen\\ Denmark}
\email{adoron@math.ku.dk}

\author[C. Linden]{Christopher Linden}
\address{Department of Mathematics \\ University of Illinois at Urbana-Champaign\\ Urbana\\ IL \\ USA}
\email{clinden2@illinois.edu}


\thanks{The first author was supported by an NSF grant DMS-1900916 and by the European Union's Horizon 2020 Marie Sklodowska-Curie grant No 839412.}

\maketitle

\begin{abstract}
In this paper we show that every non-cycle finite transitive directed graph has a Cuntz-Krieger family whose WOT-closed algebra is $B(\H)$. This is accomplished through a new construction that reduces this problem to in-degree $2$-regular graphs, which is then treated by applying the periodic Road Coloring Theorem of B\'eal and Perrin. As a consequence we show that finite disjoint unions of finite transitive directed graphs are exactly those finite graphs which admit self-adjoint free semigroupoid algebras.
\end{abstract}

\section{Introduction}

One of the many instances where non-self-adjoint operator algebra techniques are useful is in distinguishing representations of C*-algebras up to unitary equivalence. By work of Glimm we know that classifying representations of non-type-I C*-algebras up to unitary equivalence cannot be done with countable Borel structures \cite{Gli61}. Hence, in order to distinguish representations of Cuntz algebra $\O_n$, one either restricts to a tractable subclass or weakens the invariant. By restricting to permutative or atomic representations, classification was achieved by Bratteli and Jorgensen in \cite{BJ99} and by Davidson and Pitts in \cite{DP99}.

Since general representations of $\O_n$ are rather unruly, one can weaken unitary equivalence by considering isomorphism classes of not-necessarily-self-adjoint free semigroup algebras, which are WOT-closed operator algebras generated by the Cuntz isometries of a given representation of $\O_n$. The study of free semigroup algebras originates from the work of Popescu on his non-commutative disc algebra \cite{Pop96}, and particularly from work of Arias and Popescu \cite{AP00}, and of Davidson and Pitts \cite{DP98, DP99}. This work was subsequently used by Davidson, Katsoulis and Pitts to establish a general non-self-adjoint structure theorem for any free semigroup algebra \cite{DKP01} which can be used to distinguish many representations of Cuntz algebra $\O_n$ via non-self-adjoint techniques. 

The works of Bratteli and Jorgensen on iterated function systems were eventually generalized, and classification of Cuntz-Krieger representations of directed graphs found use in the work of Marcolli and Paolucci \cite{MP11} for producing wavelets on Cantor sets, and in work of Bezuglyi and Jorgensen \cite{BJ15} where they are associated to one-sided measure-theoretic dynamical systems called ``semi-branching function systems''. 

Towards establishing a non-self-adjoint theory for distinguishing representations of directed graphs, and by building on work of many authors \cite{MS11, DLP05, JK05, KK05, Ken11, Ken13, KP04}, the first author together with Davidson and B. Li extended the theory of free semigroup algebras to classify representations of directed graphs via non-self-adjoint techniques \cite{DDL20}.

\begin{definition}
Let $G=(V,E,r,s)$ be a directed graph with range and source maps $r,s: E \rightarrow V$. A family $\rS = (S_v,S_e)_{v\in V,e\in E}$ of operators on Hilbert space $\mathcal{H}$ is a \emph{Toeplitz-Cuntz-Krieger} (TCK) family if
\begin{enumerate}
\item[1.]
$\{S_v\}_{v\in V}$ is a set of pairwise orthogonal projections;

\item[2.]
$S_e^*S_e = S_{s(e)}$ for every $e\in E$;

\item[3.]
$\sum_{e\in F} S_eS_e^* \leq S_v$ for every finite subset $F\subseteq r^{-1}(v)$.
\end{enumerate}
We say that $\rS$ is a \emph{Cuntz-Krieger} (CK) family if additionally
\begin{enumerate}
\item[4.]
$\sum_{e \in r^{-1}(v)}S_eS_e^* = S_v$ for every $v\in V$ with $0 < |r^{-1}(v)| < \infty$.
\end{enumerate}
We say $\rS$ is a \textit{fully-coisometric} family if additionally
\begin{enumerate}
\item[5.]
$\sotsum_{e \in r^{-1}(v)} S_e S_e^* = S_v$ for every $v \in V$.
\end{enumerate}
\end{definition}

Given a TCK family $\rS$ for a directed graph $G$, we say that $\rS$ is \emph{fully supported} if $S_v \neq 0$ for all $v\in V$. When $\rS$ is not fully supported, we may induce a subgraph $G_{\rS}$ on the support $V_{\rS} = \{ \ v \in V \ | \ S_v \neq 0 \ \}$ of $\rS$ so that $\rS$ is really just a TCK family for the smaller graph $G_{\rS}$. Thus, if we wish to detect some property of $G$ from a TCK family $\rS$ of $G$, we will have to assume that $\rS$ is fully supported. When $G$ is transitive and $S_w \neq 0$ for some $w \in V$ it follows that $\rS$ is fully supported.

Studying TCK or CK families amounts to studying representations of Toeplitz-Cuntz-Krieger and Cuntz-Krieger C*-algebras. More precisely, let $\T(G)$ and $\O(G)$ be the universal C*-algebras generated by TCK and CK families respectively. Then representations of TCK or CK C*-algebras are in bijection with TCK or CK families respectively. The C*-algebra $\O(G)$ is the well-known graph C*-algebra of $G$, which generalizes the Cuntz--Krieger algebra introduced in \cite{CK80} for studying subshifts of finite type. We recommend \cite{Raeb05} for further preliminaries on TCK and CK families, as well as C*-algebras associated to directed graphs. 

When $G=(V,E,r,s)$ is a directed graph, we denote by $E^{\bullet}$ the collection of finite paths $\lambda = e_1 ... e_n$ in $G$ where $s(e_i) = r(e_{i+1})$ for $i=1,...,n-1$. In this case we say that $\lambda$ is of length $n$, and we regard vertices as paths of length $0$. Given a TCK family $\rS = (S_v,S_e)$ and a path $\lambda = e_1 ... e_n$ we define $S_{\lambda} = S_{e_1} \circ ... \circ S_{e_n}$. We extend the range and source maps of paths $\lambda = e_1...e_n$ by setting $r(\lambda) := r(e_1)$ and $s(\lambda) := s(e_n)$, and for a vertex $v \in V$ considered as a path we define $r(v) = v = s(v)$. A path $\lambda$ of length $|\lambda| >0$ is said to be a cycle if $r(\lambda) = s(\lambda)$. We will often not mention the range and source maps $r$ and $s$ in the definition of a directed graphs, and understand them from context.

\begin{hypothesis}\label{h:1}
Throughout the paper we will assume that whenever $\rS=(S_v,S_e)$ is a TCK family, then $\sotsum_{v\in V} S_v = I_{\H}$. In terms of representations of the C*-algebras this is equivalent to requiring that all $*$-representations of our C*-algebras $\T(G)$ and $\O(G)$ are non-degenerate.
\end{hypothesis}

\begin{definition}
Let $G=(V,E)$ be a directed graph, and let $\rS = (S_v,S_e)$ be a TCK family on a Hilbert space $\mathcal{H}$. The WOT-closed algebra $\fS$ generated by $\rS$ is called a \emph{free semigroupoid algebra} of $G$.
\end{definition}

The main purpose of this paper is to characterize which finite graphs admit self-adjoint free semigroupoid algebras. For the $n$-cycle graph, we know from \cite[Theorem 5.6]{DDL20} that $M_n(L^{\infty}(\mu))$ is a free semigroupoid algebra when $\mu$ is a measure on the unit circle $\bT$ which is not absolutely continuous with respect to Lebesgue measure $m$ on $\bT$. Thus, an easy example for a self-adjoint free semigroupoid algebra for the $n$-cycle graph is simply $M_n(\mathbb{C})$, by taking some Dirac measure $\mu = \delta_z$ for $z\in \bT$. 

On the other hand, for non cycle graphs examples which are self-adjoint are rather difficult to construct, and the first example showing this is possible was provided by Read \cite{Rea05} for the graph with a single vertex and two loops. More precisely, Read shows that there are two isometries $Z_1, Z_2$ on a Hilbert space $\H$ with pairwise orthogonal ranges that sum up to $\H$ such that the WOT-closed algebra generated by $Z_1,Z_2$ is $B(\H)$. Read's proof was later streamlined and simplified by Davidson in \cite{Dav06}.

\begin{definition}
Let $G = (V,E)$ be a directed graph. We say that $G$ is 
\begin{enumerate}
\item
\emph{transitive} if there is a path between any one vertex and another. 
\item
\emph{aperiodic} if for any two vertices $v,w \in V$ there is a $K_0$ such that any length $K\geq K_0$ can occur as the length of a path from $v$ to $w$.
\end{enumerate}
\end{definition}

Notice that for \emph{finite} (both finitely many vertices and edges) transitive graphs, the notions of a CK family and fully-coisometric TCK family coincide. In \cite[Theorem 4.3 \& Corollary 6.13]{DDL20} restrictions were found on graphs and TCK families so as to allow for self-adjoint examples.

\begin{theorem}[Theorem 4.3 \& Corollary 6.13 in \cite{DDL20}] \label{t:sa-rest}
Let $\fS$ be a free semigroupoid algebra generated by a TCK family $\rS$ of a directed graph $G = (V,E)$ such that $S_v \neq 0$ for all $v\in V$. If $\fS$ is self-adjoint then
\begin{enumerate}
\item
$\rS$ must be fully-coisometric, and;
\item
$G$ must be a disjoint union of transitive components.
\end{enumerate}
\end{theorem}

Showing that non-cycle transitive graphs other than the single vertex with two loops admit a self-adjoint free semigroupoid algebra required new ideas from directed graph theory.

\begin{definition}
Let $G = (V,E)$ be a transitive, finite and in-degree $d$-regular graph. A \emph{strong edge coloring} $c:E \rightarrow \{1,...,d\}$ is one where $c(e)\neq c(f)$ for any two distinct edges $e,f \in r^{-1}(v)$ and $v\in V$.
\end{definition}

Whenever $G = (V,E)$ has a strong edge coloring $c$, it induces a labeling of finite paths $c : E^{\bullet} \rightarrow \mathbb{F}^+_d$ which is defined for $\lambda = e_1...e_n$ via $c(\lambda) = c(e_1) ...  c(e_n)$. Since $c$ is a strong edge coloring, whenever $w \in V$ is a vertex and $\gamma = i_1 ... i_n \in \mathbb{F}^+_d$ with $i_j \in \{1,...,d\}$ is a word in colors, we may inductively construct a path $\lambda = e_1... e_n$ such that $c(e_j) = i_j$ and $r(e_1) = w$. In this way, every vertex $w$ and a word in colors uniquely define a back-tracked path whose range is $w$.

\begin{definition}
Let $G = (V,E)$ be a transitive, finite and in-degree $d$-regular graph. A strong edge coloring is called \emph{synchronizing} if for some vertex $v \in V$ there is a word $\gamma_v \in \mathbb{F}^+_d$ in colors $\{1,...,d\}$ such that for any other vertex $w\in V$, the unique back-tracked path $\lambda$ with range $w$ and color $c(\lambda) = \gamma_v$ has source $s(\lambda) = v$.
\end{definition}

It is easy to see that if a finite in-degree regular graph has a \emph{synchronizing} strong edge coloring then it is aperiodic. The converse of this statement is a famous conjecture made by Adler and Weiss in the late 60s \cite{AW70}. This conjecture was eventually proven by Trahtman \cite{Tra09} and is now called the Road Coloring Theorem. The Road Coloring Theorem was the key device that enabled the construction of self-adjoint free semigroupoid algebras for in-degree regular aperiodic directed graphs in \cite[Theorem 10.11]{DDL20}.

\begin{definition}
Let $G=(V,E)$ be a transitive directed graph. We say that $G$ \emph{has period $p$} if $p$ is the largest integer such that we can partition $V=\sqcup_{i =1}^p V_i$ so that when $e \in E$ is an edge with $s(e) \in V_i$ then $r(e) \in V_{i+1}$ (here we identify $p+1 \equiv 1$). This decomposition is called the \emph{cyclic decomposition} of $G$, and the sets $\{V_i\}$ are called the \emph{cyclic components} of $G$.
\end{definition}

\begin{remark}
In a transitive graph $G$ every vertex $v \in V$ has a cycle of finite length through it. Hence, if $G$ does not have finite periodicity this would imply that any cycle around $v$ must have arbitrarily large length. Hence, transitive graphs have finite periodicity.
\end{remark}

It is a standard fact that $G$ is $p$-periodic exactly when for any two vertices $v,w \in V$ there exists $0 \leq r < p$ and $K_0$ such that for any $K \geq K_0$ the length $pK + r$ occurs as the length of a path from $v$ to $w$, while $pK +r'$ does not occur for any $K$ and $0 \leq r' < p$ with $r\neq r'$. Hence, $G$ is $1$-periodic if and only if it is aperiodic. We will henceforth say that $G$ is periodic when $G$ is $p$-periodic with period $p \geq 2$. For a transitive directed graph, the period $p$ is also equal to the greatest common divisor of the lengths of its cycles. This equivalent definition of periodicity of a transitive directed graph is the one most commonly used in the literature.

In \cite[Question 10.13]{DDL20} it was asked whether periodic in-degree $d$-regular finite transitive graphs with $d\geq 3$ have a self-adjoint free semigroupoid algebra. In this paper we answer this question in the affirmative. In fact, we are able to characterize exactly which finite graphs have a self-adjoint free semigroupoid algebra. A generalization of the road coloring for periodic in-degree regular graphs proven by B\'eal and Perrin \cite{BP14} is then the replacement for Trahtman's aperiodic Road Coloring Theorem when the graph is periodic.

\begin{theorem} \label{t:safsa}
Let $G = (V,E)$ be a finite graph. There exists a fully supported CK family $\rS=(S_v,S_e)$ which generates a \emph{self-adjoint} free semigroupoid algebra $\fS$ if and only if $G$ is the union of transitive components.

Furthermore, if $G$ is transitive and not a cycle then $B(\H)$ is a free semigroupoid algebra for $G$ where $\H$ is a separable infinite dimensional Hilbert space.
\end{theorem}

This paper is divided into four sections including this introduction. In Section \ref{s:prc} we translate the periodic Road Coloring Theorem of B\'eal and Perrin to a more concrete statement that we end up using. In Section \ref{s:B(H)-fsa} a reduction to graphs with in-degree at least $2$ at every vertex is made, and our new construction reduces that case to the in-degree $2$-regular case. Finally in Section \ref{s:main-theorem} we combine everything together for a proof of Theorem \ref{t:safsa} and give some concluding remarks.

\section{Periodic Road coloring} \label{s:prc}

The following is the generalization of the notion of synchronization to $p$-periodic finite graphs that we shall need.

\begin{definition} \label{d:p-synch}
Let $G = (V,E)$ be a transitive, finite and in-degree $d$-regular $p$-periodic directed graph with cyclic decomposition $V=\sqcup_{i =1}^p V_i$. A strong edge coloring $c$ of $G$ with $d$ colors is called \emph{$p$-synchronizing} if there exist
\begin{enumerate}
\item distinguished vertices $v_i \in V_i$ for each $1 \leq i \leq p$, and;
\item a word $\gamma$ such that for any $1 \leq i \leq p$ and $v \in  V_i$, the unique backward path $\lambda_v$ with $r(\lambda_v) =v$ and $c(\lambda_v) = \gamma$ has source $v_i$.
\end{enumerate}
Such a $\gamma$ is called a $p$-synchronizing word for the tuple $(v_1,...,v_p)$.
\end{definition}

In order to show that every in-degree $d$-regular $p$-periodic graph is $p$-synchronizing we will translate a periodic version of the Road Coloring Theorem due to B\'eal and Perrin \cite{BP14}. We warn the reader that the graphs we consider here are in-degree regular whereas in \cite{BP14} the graphs are out-degree regular. Thus, so as to fit our choice of graph orientation, we state their definitions and theorem with edges having reversed ranges and sources.

Let $G=(V,E)$ be a transitive, finite, in-degree $d$-regular graph. If $c: E \rightarrow \{1,...,d\}$ is a strong edge coloring, each word $\gamma \in \bF_d^+$ in colors gives rise to a map $\gamma : V \rightarrow V$ defined as follows. For $v\in V$ let $\lambda_v$ be the unique path with $r(\lambda_v) =v$ whose color is $c(\lambda) = \gamma$. Then we define $v \cdot \gamma := s(\lambda_v)$. In this way, we can apply the function $\gamma$ to each subset $I \subseteq V$ to obtain another subset of vertices $I \cdot \gamma = \{ \ v \cdot \gamma \ | \ v \in I \ \}$. 

\begin{definition}
Let $G=(V,E)$ be a transitive, finite and in-degree $d$-regular with a strong edge coloring $c : E \rightarrow \{1,...,d\}$.
\begin{enumerate}
\item
We say that a subset $I$ is a \emph{$c$-image} if there exists a word $\gamma$ such that $V \cdot \gamma = I$.
\item
A $c$-image $I \subseteq V$ is called \emph{minimal} if there is no $c$-image with smaller cardinality.
\end{enumerate}
We define the \emph{rank} of $c$ to be the size of a minimal $c$-image.
\end{definition}

Note that the rank of a transitive graph is always well-defined, since any two minimal $c$-images have the same cardinality. Next, we explain some of the language used in the statement of B\'eal and Perrin's \cite[Theorem 6]{BP14}.

A (finite) \emph{automaton} is a pair $(G,c)$ where $G=(V,E)$ is a finite directed graph and $c :E \rightarrow \{1,...,d\}$ some labeling. We say that $(G,c)$ be a \emph{complete deterministic} automaton if $c|_{r^{-1}(v)}$ is bijective for each $v\in V$. This forces $G$ to be in-degree $d$-regular with a strong edge coloring $c$. We say that an automaton is \emph{irreducible} if its underlying graph is irreducible. Finally, we say that two automata $(G,c)$ and $(H,d)$ are \emph{equivalent} if they have isomorphic underlying graphs. The statement of \cite[Theorem 6]{BP14} then says that any irreducible, complete deterministic automaton $(G,c)$ is equivalent to a complete deterministic automaton whose rank is equal to the period of $G$. This leads to the following restatement of the periodic Road Coloring Theorem of B\'eal and Perrin \cite[Theorem 6]{BP14} in the language of directed graphs and their colorings.

\begin{theorem}[Theorem 6 of \cite{BP14}]
Let $G=(V,E)$ be a transitive, finite and in-degree $d$-regular graph. Then $G$ is $p$-periodic if and only if there exists a strong edge coloring $c : E \rightarrow \{1,...,d\}$ with rank $p$.
\end{theorem}

As a corollary, we have that every finite in-degree $d$-regular graph is $p$-synchronizing (in the sense of Definition \ref{d:p-synch}) when its period is $p$. This will be useful to us in the next section.

\begin{corollary} \label{c:p-synch-exists}
Let $G=(V,E)$ be a transitive, finite, in-degree $d$-regular and $p$-periodic graph with cyclic decomposition $V=\sqcup_{i =1}^p V_i$. Then there is a strong edge coloring $c : E \rightarrow \{1,...,d\}$ which is $p$-synchronizing.

Furthermore, if $\gamma$ is a $p$-synchronizing word for $(v_1,...,v_p)$ and $w_i \in V_i$ is some vertex for some $1 \leq i \leq p$, then there are vertices $w_j \in V_j$ for $j \neq i$ and a $p$-synchronizing word $\mu$ for $(w_1,...,w_p)$.
\end{corollary}

\begin{proof}
Let $c$ be a strong edge coloring with a minimal $c$-image of size $p$. This means that there is a word $\gamma \in \bF_d^+$ such that $|V \cdot \gamma| = p$. If $\gamma$ is not of length which is a multiple of $p$, we may concatenate $\gamma'$ to $\gamma$ to ensure that $|\gamma \gamma'| = kp$ for some $0 \neq k \in \mathbb{N}$. Since $V \cdot \gamma$ is minimal we have that $V \cdot \gamma \gamma'$ is also of size $p$. Hence, without loss of generality we have that $|\gamma| = kp$ for some $0 \neq k \in \mathbb{N}$. Since $|\gamma| = kp$, we see that the function $\gamma : V \rightarrow V$ must send elements in $V_i$ to elements in $V_i$. Since each $V_i$ is non-empty and $|V \cdot \gamma| = p$, there is a unique $v_i \in V \cdot \gamma$ such that $V_i \cdot \gamma = \{v_i\}$. It is then clear that $(v_1,...,v_p)$ together with $\gamma$ show that $c$ is $p$-synchronizing.

For the second part, without loss of generality assume that $i=1$, and let $\lambda$ be a path with range $v_1$ and source $w_1$, whose length must be a multiple of $p$. Then for $\mu := \gamma \cdot c(\lambda)$, we still have $|V \cdot \mu| = p$ from minimality of $V \cdot \gamma$, and also $v_1 \cdot c(\lambda) = w_1$. Thus, we see that as in the above proof there are some $w_j \in V_j$ with $j\neq 1$ such that $\mu$ is $p$-synchronizing for $(w_1,...,w_p)$.
\end{proof}

\section{$B(\H)$ as a free semigroupoid algebra} \label{s:B(H)-fsa}

In \cite{Dav06} Read isometries $Z_1,Z_2$ with additional useful properties are obtained on a separable infinite dimensional Hilbert space $\H$. More precisely, from the proof of \cite[Lemma 1.6]{Dav06} we see that there are orthonormal bases $\{h_j \}_{j \in \bN}$ and $\{g_i \}_{i \in \bN}$ together with a sequence 
$$
S_{i,j,k} \in \Span \{ \ Z_w \ | \ w\in \mathbb{F}_2^+ \ \text{with} \ |w| = 2^k \ \}
$$ 
such that $S_{i,j,k}$ converges WOT to the rank one operator $g_i \otimes h_j^*$. As a consequence of this, in \cite[Theorem 1.7]{Dav06} it is shown that the WOT closed algebra generated by $\{ \ Z_w \ | \ \mu \in \mathbb{F}_2^+ \ \text{with} \ |w| = 2^k \ \}$ is still $B(\H)$. That $2^k$ above can be replaced with any non-zero $p \in \mathbb{N}$ was claimed after \cite[Question 10.13]{DDL20}, and we provide a proof for it here. 

\begin{proposition}\label{p:pgen} For any non-zero $p \in \mathbb{N}$, we have that the WOT-closed algebra $\fZ_p$ generated by $\{ \ Z_{\mu} \ | \ \mu \in \mathbb{F}_2^+ \ \text{with} \  |\mu| =p \ \}$ is $B(\H)$.
\end{proposition}

\begin{proof}
As there are finitely many residue classes modulo $p$, there is some $m$ such that $p$ divides $2^k + m$ for infinitely many $k$. Pick a word $w$ with length $|w|=m$, and note that $S_{i,j,k}Z_w \in \Alg \{ \ Z_{\mu} \ | \ \mu \in \mathbb{F}_2^+ \ \text{with} \  |\mu| =p \ \}$ for infinitely many $k$. Thus, the rank one operator $(g_i \otimes h_j^*) Z_w$ is in $\fZ_p$. Since any operator in $B(\H)$ is the WOT limit of finite linear combinations of operators of the form $g_i \otimes h_j^*$, we see that $A = AZ^*_w Z_w \in \fZ_p$ for any $A \in B(\H)$. Hence, $\fZ_p = B(\H)$.
\end{proof}

Next we reduce the problem of showing that $B(\H)$ is a free semigroupoid algebra for a transitive graph $G$ to a problem about vertex corners.

\begin{lemma} \label{l:cornerB(H)}
Let $G$ be a transitive graph. Suppose $\rS$ is a TCK family on $\H$ such that for any $v \in V$ there exists $w\in V$ such that $S_v \fS S_w = S_v B(\H) S_w$. Then $\fS = B(\H)$.
\end{lemma}

\begin{proof}
Let $v',w' \in V$ be arbitrary vertices. By assumption, there is $w \in V$ such that $S_{v'}\fS S_w = S_{v'}B(\H)S_w$. Let $\lambda$ be a path from $w'$ to $w$, and let $B \in B(\H)$. Then $S_{v'}BS_{\lambda} \in S_{v'} \fS S_{w'}$ for any $B\in B(\H)$. Let $B = A S_{\lambda}^*$ for general $A\in B(\H)$ so that $S_{v'}BS_{\lambda} = S_{v'} A S_{w'} \in S_{v'} \fS S_{w'}$. Hence, we obtain that $S_{v'} \fS S_{w'} = S_{v'} B(\H) S_{w'}$ for any $v',w' \in V$. Since $\sotsum_{v\in V} S_v = I_{\H}$ we get that $\fS = B(\H)$.
\end{proof}

Let $G = (V,E)$ be a finite directed graph. For an edge $e \in E$ we define the \emph{edge contraction} $G/e$ of $G$ by $e$ to be the graph obtained by removing the edge $e$ and identifying the vertices $s(e)$ and $r(e)$. When $r(e) \neq s(e)$, our convention will be that the identification of $r(e)$ and $s(e)$ is carried out by removing the vertex $r(e)$, and every edge with source / range $r(e)$ is changed to have source / range $s(e)$ respectively.
 
  \begin{lemma}\label{l:edge-cont}
  Let $G = (V,E)$ be a transitive and finite directed graph which is not a cycle. Let $e_0 \in E$ such that $r(e_0)$ has in-degree $1$. Then:
  \begin{enumerate}
  \item The edge $e_0$ is not a loop.
  \item The edge contraction $G/e_0$ has one fewer vertex of in-degree $1$ than $G$ has.
  \item The edge contraction $G/e_0$ is a finite transitive directed graph which is not a cycle.
  \end{enumerate}
  \end{lemma}
	
  \begin{proof}
  If $e_0$ were a loop, then the assumptions that $G$ is transitive and that $r(e_0)$ has in-degree $1$ would imply that $G$ is a cycle with one vertex. This contradicts our assumption that $G$ is not a cycle, so (i) is proved.
  
  Since now we must have $r(e_0) \neq s(e_0)$, we adopt our convention discussed above. Since $r(e_0)$ has in-degree $1$, there are no other edges whose range is $r(e_0)$, so contraction does not change the range of any of the surviving edges. In particular, the in-degree of the remaining vertices is unchanged. Since we have removed a single vertex of in-degree $1$, (ii) is proved.
  
By construction $G/e_0 $ is a finite directed graph. It is straightforward to verify that transitivity of $G$ implies transitivity of $G/e_0$. A finite, transitive, directed graph is a cycle if and only if every vertex has in-degree $1$. Since $G$ is by assumption not a cycle, it has a vertex of in-degree at least $2$. Since the in-degree of the remaining vertices is unchanged, $G/e_0$ also has a vertex of in-degree at least $2$. Hence $G/e_0$ is not a cycle and (iii) is proved.
  \end{proof}
	
The following proposition shows that such edge contractions preserve the property of having $B(\H)$ as a free semigroupoid algebra.

  \begin{proposition} \label{p:edge-cont-FSA}
Let $G = (V,E)$ be a transitive and finite directed graph that is not a cycle, and let $e_0 \in E$ such that $ r(e_0)$ has in-degree $1$. If $G/e_0$ has $B(\H)$ as a free semigroupoid algebra, then so does $G$.
\end{proposition}

\begin{proof}
Let $v_0 := r(e_0)$ so that  $G/e_0 = (\tilde{V}, \tilde{E}) = (V \setminus \{v_0\}, E \setminus \{ e_0\})$. 
Let $\widetilde{\rS} = (\widetilde{S}_{v}, \widetilde{S}_{e})$ be a TCK family for $G/e$ on $\H$ such that $\widetilde{\fS} = B(\H)$. By Theorem \ref{t:sa-rest} we get that $\widetilde{\rS}$ is actually a CK family. Write $\H = \bigoplus_{v  \in \tilde{V}} \widetilde{S}_{v} \H$ and let $\H_v = \widetilde{S}_{v} \H$ for $v \in \tilde{V}$. Let $\H_{v_0}$ be a Hilbert space identified with $\widetilde{S}_{s(e_0)}\H$ via a fixed unitary identification $J : \H_{v_0} \rightarrow \widetilde{S}_{{s(e_0)}}\H$ and form the space $\K = \bigoplus_{v\in V} \H_v$. We define a CK family $\rS$ for $G$ on $\K$ as follows: Let $S_v$ be the projection onto $\H_v$ for each $v\in V$. For edges $e \in \tilde{E}$ with $s(e)\neq v_0$ we extend linearly the rule 
$$
S_e \xi = \begin{cases} 
\widetilde{S}_{e} \xi & \text{ when }  \xi \in \H_{s(e)}, \\
0 & \text{ when }  \xi \in \H_{s(e)}^{\perp},
\end{cases}
$$
for edges $e\in \tilde{E}$ with $s(e) = v_0$ we extend linearly the rule
$$
S_e \xi = \begin{cases} 
\widetilde{S}_{e}J \xi & \text{ when }  \xi \in H_{v_0}, \\
0 & \text{ when }  \xi \in H_{v_0}^{\perp},
\end{cases}
$$
and finally for $e_0$ we extend linearly the rule
$$
S_{e_0} \xi = \begin{cases} 
J^* \xi & \text{ when }  \xi \in \H_{s(e_0)}, \\
0 & \text{ when }  \xi \in \H_{s(e_0)}^{\perp}.
\end{cases}
$$ 
Since $J$ is a unitary, and since $\widetilde{\rS}$ is a CK family for $\widetilde{G}$, we see that $\rS = (S_v,S_e)$ is a CK family for $G$. 

We verify that for any vertex $v \in V$ we have $S_v \fS S_v = S_v B(\K) S_v$. First note that for any $v \in V$ we have 
$$
S_v \fS S_v = \overline{\Span}^{\wot}\{ \  S_{\lambda} \ | \  r(\lambda) = v = s(\lambda), \ \lambda \in E^{\bullet} \ \},
$$
and similarly for every $v\in \tilde{V}$ we have a description as above for $\widetilde{S}_v \widetilde{\fS} \widetilde{S}_v$ in terms of cycles in $G / e_0$.

Suppose now that $v \in \tilde{V}$ and that $\lambda$ is a cycle through $v$ in $G$. If $\lambda$ does not go through $v_0$, then $S_v S_{\lambda} S_v = S_v \widetilde{S}_{\lambda} S_v \in S_v \widetilde{\fS} S_v$. Next, if $\lambda = \mu \nu$ is a simple cycle such that $s(\mu) = v_0$, since $e_0$ is the unique edge with $r(e_0) = v_0$, we may write $\nu = e_0 \nu'$. We then get that
$$
S_v S_{\lambda} S_v = S_v S_{\mu}J^* S_{\nu'} S_v = S_v \widetilde{S}_{\mu} \widetilde{S}_{\nu'} S_v \in S_v \widetilde{\fS} S_v.
$$
For a general cycle $\lambda$ around $v$ which goes through $v_0$, we may decompose it as a concatenation of simple cycles, and apply the above iteratively to eventually get that $S_v S_{\lambda} S_v \in S_v\widetilde{\fS} S_v$. Hence, for $v \in \tilde{V}$ we have  
\begin{equation*}
S_{v} \fS S_{v} = S_{v} \widetilde{\fS} S_{v} = S_{v} B(\H)S_{v} = S_v B(\K) S_v.
\end{equation*}
Finally, for $v = v_0$ fix $\mu$ some cycle going through $v_0$ in $G$, and write $\mu =e_0 \mu'$. For any $\lambda$ which is a cycle in $G$ going through $s(e_0)$, we have that 
$$
J^*S_{\lambda}JS_{\mu} = S_{e_0} S_{\lambda} S_{\mu'} \in S_{v_0} \fS S_{v_0}.
$$ 
Hence, from our previous argument applied to $s(e_0) \neq v_0$, we get
\begin{equation} \label{eq:supset-cycle}
J^* S_{s(e_0)} B(\H) S_{s(e_0)} J S_{\mu} = J^* S_{s(e_0)} \fS S_{s(e_0)} J S_{\mu} \subseteq S_{v_0} \fS S_{v_0}.
\end{equation}
Next, for $B\in S_{v_0} B(\K) S_{v_0}$ we take $A= S_{s(e_0)} J B S_{\mu}^*J^* S_{s(e_0)}$ which is now in $S_{s(e_0)} B(\H) S_{s(e_0)}$, so that 
$$
S_{s(e_0)} J B S_{v_0} = A S_{s(e_0)} J S_{\mu} \in S_{s(e_0)} B(\H) S_{s(e_0)} J S_{\mu}.
$$
By varying over all $B\in S_{v_0} B(\K) S_{v_0}$ and multiplying by $J^*$ on the left, we get that
\begin{equation*}
S_{v_0} B(\K) S_{v_0} \subseteq J^* S_{s(e_0)} B(\H) S_{s(e_0)} J S_{\mu}.
\end{equation*}
Thus, with equation \eqref{eq:supset-cycle} we get that $S_{v_0} B(\K)S_{v_0} = S_{v_0} \fS S_{v_0}$. Now, since for any $v\in V$ we have that $S_{v} \fS S_{v} = S_v B(\K) S_v$, by Lemma \ref{l:cornerB(H)} we are done.
\end{proof}

Hence, edge contraction together with Lemma \ref{l:edge-cont} can be repeatedly applied to any finite transitive directed graph $G$ which is not a cycle in order to obtain another such graph $\widetilde{G}$ which has in-degree at least $2$ for every vertex. By applying Proposition \ref{p:edge-cont-FSA} to this procedure, we obtain the following corollary.

\begin{corollary} \label{c:in-deg-2}
Let $G$ be a transitive and finite directed graph which is not a cycle, and let $\widetilde{G}$ be a graph resulting from repeatedly applying edge contractions to edges $e\in E$ with $r(e)$ of in-degree $1$. Then $\widetilde{G}$ has in-degree at least $2$ for every vertex, and if $\widetilde{G}$ has $B(\H)$ as a free semigroupoid algebra, then so does $G$.
\end{corollary}

Let $G = (V,E)$ is finite directed graph which is transitive and has in-degree at least $2$ at every vertex. For a vertex $v\in V$ with in-degree $d_v \geq 3$ we define the \emph{$v$-lag} of $G$ to be the graph $\widehat{G}_v = (\widehat{V}_v, \widehat{E}_v)$ obtained as follows: all vertices beside $v$ and edges ranging in such vertices remain the same. We list $(u_0,...,u_{d_v-1})$ the tuple of vertices in $G$ which are the source of an edge with range $v$, counted with repetition so that there is a unique edge $e_j$ from each $u_j$ to $v$. We add $d_v-2$ vertices $v_1,...,v_{d_v-2}$ and set an edge $f_i$ from $v_i$ to $v_{i-1}$ when $2 \leq i \leq d_v-2$ and an edge $f_1$ from $v_1$ to $v$. We then replace each edge $e_j$ from $u_j$ to $v$ in $G$ with an edge $\hat{e}_j$ from $u_j$ to $v_j$ when $1 \leq j \leq d_v-2$, replace an edge $e_0$ from $u_0$ to $v$ in $G$ by an edge $\hat{e}_0$ from $u_0$ to $v$, and replace an edge $e_{d_v-1}$ from $u_{d_v-1}$ to $v$ in $G$ with an edge $\hat{e}_{d_v-1}$ from $u_{d_v-1}$ to $v_{d_v-2}$. The resulting graph is over the vertex set $\widehat{V}_v = V \sqcup \{v_1,...,v_{d_v-2}\}$ and we denote it by $\widehat{G}_v$. The construction will replace only those edges going into $v$ with those shown in Figure \ref{f:split} (the sources $u_0,...,u_{d_v -1}$ of such edges may have repetitions), and everything else will remain the same.

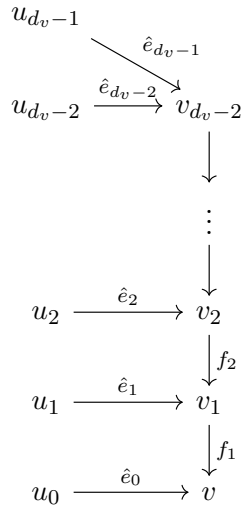
\begin{figure}[H]
  \begin{tikzcd}
  u_{d_v-1} \arrow["\hat{e}_{d_v-1}"]{rd} & \\
   u_{d_v-2} \arrow["\hat{e}_{d_v-2}"]{r} & v_{d_v-2} \arrow[d]\\
  & \vdots  \arrow[d]\\
  u_2 \arrow["\hat{e}_2"]{r} &  v_2  \arrow["f_2"]{d} \\
   u_1 \arrow["\hat{e}_1"]{r} &  v_1 \arrow["f_1"]{d}  \\
    u_0  \arrow["\hat{e}_0"]{r} & v     
  \end{tikzcd}
	\caption{The lag of $G$ applied at $v$.} \label{f:split}
\end{figure} 

\begin{lemma} \label{l:transitive-less-in-deg-3}
Let $G=(V,E)$ be a transitive and finite directed graph with in-degree at least $2$ at every vertex. Let $v \in V$ be a vertex with in-degree $d_v \geq 3$. Then $\widehat{G}_v$ is a finite transitive directed graph with in-degree at least $2$ at every vertex, and has one fewer vertex of in-degree at least $3$.
\end{lemma}

\begin{proof}
We see from the construction that $v_1,..., v_{d_v-2}$, as well as $v$, all have in-degree $2$. So we have reduced by one the number of vertices of in-degree at least $3$, and all other vertices except for $v_1,...,v_{d_v -2}$ and $v$ still have the same in-degree. Hence, $\widehat{G}_v$ has one fewer vertex with in-degree at least $3$.

We next show that $\widehat{G}_v$ is transitive. Indeed, every one of the vertices $\{v_1,..., v_{d_v-2}\}$ leads to $v$, and the set of vertices $\{u_1,...,u_{d_v-1}\}$ lead to $v_1,...,v_{d_v-2}$ so it would suffice to show that for any two vertices $w,u \in V$ we have a path in $\widehat{G}_v$ from $w$ to $u$. Since $G$ is transitive, we have a path $\lambda = g_1...g_n$  in $G$ from $w$ to $u$. Next, define $\hat{g}_k$ to be $g_k$ if $g_k \neq e_j$ for all $j$ and whenever $g_k = e_j$ for some edge $e_j$ with range in $v$ we set $g_k:= f_1...f_j\hat{e}_j$ which is a path in $\widehat{G}_v$ from $u_j$ to $v$. This way the new path $\hat{g} = \hat{g}_1....\hat{g}_n$ is a path in $\widehat{G}_v$ from $w$ to $u$.
\end{proof}

Our construction depends on the choice of orderings for the $\{u_j\}$, so when we write $\widehat{G}_v$ we mean a fixed ordering for sources of incoming edges of the vertex $v$ in the above construction.

Next, let $G=(V,E)$ be a finite directed graph with in-degree at least $2$ at every vertex. For a vertex of in-degree at least $3$, let $\widehat{G}_v$ be the $v$-lag of $G$. We define a map $\theta$ on paths $E^{\bullet}$ of $G$ as follows: If $e\in E$ is any edge with $r(e) \neq v$, we define $\theta(e) = e$. Next, if $e_j$ is the unique edge from $u_j$ to $v$ in $G$ we define $\theta(e_j) = f_1 f_2 ... f_j \hat{e}_j$ which is a path from $u_j$ to $v$ in $\widehat{G}_v$. The map $\theta$ then extends to a map (denoted still by) $\theta$ on finite paths $E^{\bullet}$ of $G$ by concatenation, whose restriction to $V$ is the embedding of $V \subseteq \widehat{V}_v$.

\begin{lemma} \label{l:bijection}
Let $G=(V,E)$ be a transitive and finite directed graph with in-degree at least $2$ at every vertex. For a vertex $v\in V$ of in-degree at least $3$, let $\widehat{G}_v$ be the $v$-lag of $G$. Then $\theta$ is a bijection between $E^{\bullet}$ and paths in $\widehat{E}_v^{\bullet}$ whose range and source are both in $V$.
\end{lemma}

\begin{proof}
Since $\theta$ is injective on edges and vertices, it must be injective on paths as it is defined on paths by extension. 

Suppose now that $u \in V$ and $\widehat{\lambda} = \hat{g}_1 ... \hat{g}_j$ is a path in $\widehat{G}_v$ from $u$ to $v$ such that $s(\hat{g}_k) \notin V$ for $k=1,...,j-1$. Then it must be that $\widehat{\lambda} = f_1f_2...f_j\hat{e}_j$ with $s(\hat{e}_j) = u_j = u$, so that $\widehat{\lambda} = \theta(e_j)$. A general path in $\widehat{G}_v$ between two vertices in $V$ is then the concatenation of edges in $E$ and paths of the form $f_1f_2...f_j\hat{e}_j$ as above, so that $\theta$ is surjective. 
\end{proof}

Applying a $v$-lag at each vertex of in-degree $3$ repeatedly until there are no more such vertices, we obtain a directed graph $\widehat{G} = (\widehat{V},\widehat{E})$. Since the construction at each vertex of in-degree at least $3$ changes only edges going into the vertex $v$, the order in which we apply the lags does not matter, and we will get the same directed graph $\widehat{G} = (\widehat{V},\widehat{E})$. The following is a result of applying Lemma \ref{l:transitive-less-in-deg-3} and Lemma \ref{l:bijection} at each step of this process.

\begin{corollary} \label{c:in-deg-reg-2}
Let $G=(V,E)$ be a transitive and finite directed graph with in-degree at least $2$ at every vertex. Let $\widehat{G} = (\widehat{V},\widehat{E})$ be the graph resulting from repeatedly applying a $v$-lag at every vertex $v$ of in-degree at least $3$. Then $\widehat{G}$ is in-degree $2$-regular and there is an embedding $V\subseteq \widehat{V}$ which extends to a bijection $\theta$ from paths $E^{\bullet}$ to paths in $\widehat{E}^{\bullet}$ whose range and source in $V$.
\end{corollary}

Now let $G=(V,E)$ be a finite directed graph with in-degree at least $2$ at every vertex, and $\widehat{G} = (\widehat{V},\widehat{E})$ be the graph in Corollary \ref{c:in-deg-reg-2}. Given a strong edge coloring $c$ of $\widehat{G}$ with two colors $\{1,2\}$, each edge $g \in E$ inherits a labeling $\ell$ given by $\ell(g) := c(\theta(g))$ with labels $\{c(\theta(g))\}$. We then extend this labeling to paths in $G$ by setting for any path $\lambda = g_1...g_n \in E^{\bullet}$ in $G$ the label $\ell(\lambda) = c(\theta(\lambda)) = c(\theta(g_1)) ... c(\theta(g_n))$ where $\theta(\lambda) \in \widehat{E}^{\bullet}$ is the path in $\widehat{G}$ corresponding to $\lambda$ whose range and source are always in $V$.

Using the labeling $\ell$ we construct a Cuntz-Krieger family $\rS^{\ell} = (S^{\ell}_v,S^{\ell}_e)$ for our original graph $G$ as follows: let $\H$ be a separable infinite dimensional Hilbert space. Let $\K = \bigoplus_{v \in V} \H_v$ where $\H_v$ is a copy of $\H$ identified via a unitary $J_v : \H_v \rightarrow \H$. First we define $S^{\ell}_v$ to be the projection onto $\H_v$ for $v \in V$. Then, for $e \in E$ we define $S^{\ell}_e$ by linearly extending the rule
$$
S^{\ell}_e\xi =  \begin{cases} J_{r(e)}^*Z_{\ell(e)} J_{s(e)} \xi & \text{ for } \xi \in H_{s(e)}\\
0  & \text{ for } \xi \in H_{s(e)}^{\perp}.
\end{cases}
$$ 
where $Z_{\ell(e)}$ is the composition of Read isometries $Z_1$ and $Z_2$ given as follows: for each $e\in E$ there are $f_1,...,f_j \in \widehat{E}$ (or non at all when $r(e)$ has in-degree $2$ in $G$) such that $\theta(e) = f_1 f_2 ... f_j \hat{e}$ as in the iterated construction of $\widehat{G}$. Thus, we get that $Z_{\ell(e)} = Z_{c(\theta(e))} = Z_{c(f_1)} \circ ... \circ Z_{c(f_j)} \circ Z_{c(\hat{e})}$.

\begin{proposition}\label{p:CKgen}
Let $G$ be a transitive and finite directed graph such that all vertices have in-degree at least $2$, and let $\widehat{G}$ be the in-degree $2$ regular graph constructed in Corollary \ref{c:in-deg-reg-2}. Let $p$ be the period of $\widehat{G}$. Then for any strong edge coloring $c: \widehat{E} \rightarrow \{1,2\}$ for $\widehat{G}$ we have that $\rS^{\ell}$ is a CK family for $G$. Furthermore, if $c$ is $p$-synchronizing for $\widehat{G}$, then the free semigroupoid algebra $\fS^{\ell}$ generated by $\rS^{\ell}$ as above is $B(\K)$.
\end{proposition}
 
\begin{proof}
We first show that $\rS^{\ell}$ is a CK family, given a strong edge coloring $c$ for $\widehat{G}$. It is easy to show by definition that $\rS^{\ell}$ is a TCK family, so we show the condition that makes it into a CK family. Indeed, for $v\in V$ we have that
\begin{equation} \label{e:labelck}
\sum_{e\in r^{-1}(v)} S_e^{\ell}(S_e^{\ell})^* = S^{\ell}_vJ_v^* \Big( \sum_{e \in r^{-1}(v)}Z_{\ell(e)}Z_{\ell(e)}^* \Big) J_v S^{\ell}_v.
\end{equation}
Now, let the in-degree of $v$ be $d=d_v$, and let $(u_0,...,u_{d-1})$ be the sources of edges incoming to $v$ in $G$. Then in $\widehat{G}$ the path $\theta(e_j)$ (associated to the edge $e_j$ from $u_j$ to $v$ in $G$) is given by $\theta(e_j) = f_1f_2 .. f_j \hat{e}_j$. Hence, since $c$ is a strong edge coloring with two colors we see that 
$$
Z_{c(f_1...f_{d-2})}Z_{c(f_1...f_{d-2})}^* = Z_{c(\theta(e_{d-1}))} Z_{c(\theta(e_{d-1}))}^* +  Z_{c(\theta(e_{d-2}))}Z_{c(\theta(e_{d-2}))}^*
$$
as well as
$$
Z_{c(f_1...f_j)}Z_{c(f_1...f_j)}^* = Z_{c(\theta(e_j))} Z_{c(\theta(e_j))}^* +  Z_{c(f_1...f_jf_{j+1})}Z_{c(f_1...f_jf_{j+1})}^*.
$$
By applying these identities repeatedly we obtain that
$$
\sum_{e \in r^{-1}(v)}Z_{\ell(e)}Z_{\ell(e)}^* = I_{\H}.
$$
Thus from equation \eqref{e:labelck} we get that $\rS^{\ell}$ is a CK family.

Next, suppose that the strong edge coloring $c$ of $\widehat{G}$ is $p$-synchronizing. We show that the free semigroupoid algebra $\fS^{\ell}$ of $\rS^{\ell}$ for the graph $G$ is $B(\K)$. Let $v \in V$ be a vertex. By the second part of Corollary \ref{c:p-synch-exists} we have a $p$-synchronizing word $\gamma_v$ for $v \in \widehat{V}$ in the sense that whenever $u \in \widehat{V}$ is in the same cyclic component of $v$ in $\widehat{G}$, then $u \cdot \gamma_v = v$. Let $\gamma$ be a word in two colors of length divisible by $p$. There is a unique path $\widehat{\lambda}$ in $\widehat{G}$ with $r(\widehat{\lambda}) =v$ such that $c(\widehat{\lambda}) = \gamma$. Since the length of $\widehat{\lambda}$ is divisible by $p$, we have that  $s(\widehat{\lambda})$ must also be in the same cyclic component as $v$, so by the $p$-synchronizing property of $\gamma_v$ there is a unique path $\widehat{\lambda}_v$ with $c(\widehat{\lambda}_v) = \gamma_v$ and $r(\widehat{\lambda}_v) = s(\widehat{\lambda})$ and $s(\widehat{\lambda}_v) = v$. This defines a cycle $\widehat{\lambda} \widehat{\lambda}_v$ around $v$ whose color is $c(\widehat{\lambda}\widehat{\lambda}_v) = \gamma\gamma_v$. 

By Corollary \ref{c:in-deg-reg-2}, there is a unique cycle $\mu$ around $v$ in $G$ such that $\theta(\mu) = \widehat{\lambda} \widehat{\lambda}_v$. Then $\ell(\mu) = c(\widehat{\lambda}\widehat{\lambda}_v) = \gamma \gamma_v$, and we get that 
$$
S^{\ell}_{\mu} = S^{\ell}_vJ_{v}^*Z_{\gamma}Z_{\gamma_v} J_{v} S^{\ell}_v \in  S^{\ell}_v \mathfrak{S}S^{\ell}_v.
$$

Since $\widehat{\lambda}$ is a general path with range in $v$ whose length is divisible by $p$, we see that $c(\widehat{\lambda}) = \gamma$ is an arbitrary word of length divisible by $p$. Thus, by Proposition \ref{p:pgen} we obtain $S_vJ_v^*BZ_{\gamma_v} J_v S_v \in  S_v \fS S_v$ for any $B \in B(\H)$. By taking $B = A Z_{\gamma_v}^*$ we see that $S_v J_v^* A J_v S_v \in S_v \fS S_v$ for any $A \in B(\H)$. Finally, we have shown that $S_v \fS S_v = S_v B(\K) S_v$ for arbitrary $v\in V$ so that by Lemma \ref{l:cornerB(H)} we conclude that $\fS = B(\K)$.
\end{proof}

\begin{example}\label{e:embed}
In the case where the graph $G$ is a single vertex with $d\geq 3$ edges, the construction of $\widehat{G}$ yields the graph shown in Figure \ref{f:onevert}.

\begin{figure}[H]
  \begin{tikzcd}
  v_{d-2} \arrow{d}  \\
   v_{d-3} \arrow[d]\\
   \vdots  \arrow[d]\\
   v_{1} \arrow{d}  \\
     v \arrow[loop left, "e"] \arrow[u, bend right=30]  \arrow[uu, bend right=30] \arrow[uuu, bend right=30]  \arrow[uuuu, bend right=30] \arrow[uuuu, bend left=30]
  \end{tikzcd}
	\caption{Splitting a vertex with $d\geq 3$ loops.} \label{f:onevert}
\end{figure}
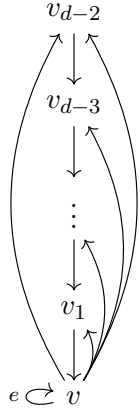

A strong $2$-coloring of the graph in Figure \ref{f:onevert} lifts to a coloring of $G$ by words in $\{1,2\}$. This determines a monomial embedding $\O_d \hookrightarrow \O_2$, where each generator for $\O_d$ is sent to the appropriate composition of the generators of $\mathcal{O}_2$. Such monomial embeddings arise and are studied in \cite{Lin20}. If we represent the canonical generators of $\O_2$ by a pair of Read isometries $Z_1,Z_2$, we obtain a representation of $\O_d$. If the coloring is synchronizing, the generating isometries of $\O_d$ will generate $B(\H)$ as a free semigroup algebra. In fact, any strong 2-coloring of the graph in Figure \ref{f:onevert} is synchronizing: if the edge labeled $e$ has color $i \in \{1,2\}$, then it is easy to see that $i^d$ is a synchronizing word for the vertex $v$.
\end{example}

\section{Self-adjoint free semigroupoid algebras} \label{s:main-theorem}

In this section we tie everything together to obtain our main theorem, and make a few concluding remarks.

\vspace{4pt}

\noindent {\bf Theorem \ref{t:safsa}} (Self-adjoint free semigroupoid algebras){\bf .} \emph{
Let $G = (V,E)$ be a finite graph. There exists a fully supported CK family $\rS=(S_v,S_e)$ which generates a \emph{self-adjoint} free semigroupoid algebra $\fS$ if and only if $G$ is the union of transitive components.}

\emph{Furthermore, if $G$ is transitive and not a cycle then $B(\H)$ is a free semigroupoid algebra for $G$ where $\H$ is a separable infinite dimensional Hilbert space.}

\vspace{0pt}

\begin{proof}
If $\fS$ is self-adjoint, Theorem \ref{t:sa-rest} tells us that $G$ must be the disjoint union of transitive components. 

Conversely, if $G$ is the union of transitive components, we will form a CK family whose free semigroupoid algebra is self-adjoint for each transitive component separately, and then define the one for $G$ by taking their direct sum. Hence, we need only show that every finite transitive graph $G$ has a self-adjoint free semigroupoid algebra. Then there are two cases:

\begin{enumerate}
\item {\bf If $G$ is a cycle of length $n$:}
By item (1) of \cite[Theorem 5.6]{DDL20} we have that $M_n(L^{\infty}(\mu))$ is a free semigroupoid algebra when $\mu$ is a measure on the unit circle $\bT$ which is not absolutely continuous with respect to Lebesgue measure $m$ on $\bT$.

\item {\bf If $G$ is not a cycle:}
By Corollary \ref{c:in-deg-2} we may assume without loss of generality that $G$ has no vertices with in-degree $1$. By the first part of Corollary \ref{c:p-synch-exists} there exists a $p$-synchronizing strong edge coloring $c$ for $\widehat{G}$, so we may apply Proposition \ref{p:CKgen} to deduce that $B(\H)$ is a free semigroupoid algebra for $G$.
\end{enumerate}
\end{proof}

\begin{remark}
Note that in the proof of Theorem \ref{p:CKgen}, since $\widehat{G}$ is always in-degree $2$-regular, we only needed the in-degree $2$-regular case of the periodic Road Coloring Theorem in Corollary \ref{c:p-synch-exists}. On the other hand, we could only prove Proposition \ref{p:pgen} for the free semigroup on two generators, so it was also necessary to reduce the problem to the in-degree $2$-regular case. The latter explains why an iterated version of Proposition \ref{p:CKgen} akin to the one for the first construction in Proposition \ref{p:edge-cont-FSA} is not so readily available. 
\end{remark}

To construct a suitable Cuntz-Krieger family as in the discussion preceding Proposition \ref{p:CKgen}, for each vertex in $G$ with in-degree at least $3$ we must choose a monomial embedding of $\O_d$ into $\O_2$. Example \ref{e:embed} shows that choosing a monomial embedding is equivalent to choosing a strong edge coloring of a certain binary tree with $d$ leaves. Constructing such a tree for each vertex gives the construction of $\widehat{G}$, and this is where the intuition for our proof originated from.

\subsection*{Acknowledgments} The first author would like to thank Boyu Li for discussions that led to the proof of Proposition \ref{p:pgen}. Both authors are grateful to the anonymous referees for pointing out issues with older proofs, and for providing suggestions that improved the exposition of the paper. Both authors are also grateful to Florin Boca and Guy Salomon for many useful remarks on previous draft versions of the paper.


\end{document}